\documentclass[12pt,english,fleqn,liststotoc,bibtotoc,idxtotoc,BCOR7.5mm,tablecaptionabove]{amsart}
\usepackage[T1]{fontenc}
\usepackage[latin1]{inputenc}
\usepackage{geometry}
\usepackage{color}
\usepackage{amsthm}
\usepackage{amstext}
\usepackage{amssymb}
\usepackage{amsmath}
\usepackage{graphicx}
\usepackage{young}
\usepackage{xypic}
\usepackage{amscd}
\usepackage[section]{placeins}
\usepackage[svgnames]{xcolor}
\usepackage[unicode=true,
bookmarks=true,bookmarksnumbered=true,bookmarksopen=false,
breaklinks=false,pdfborder={0 0 1},backref=false,colorlinks=true]
{hyperref}
\hypersetup{pdftitle={},
	pdfauthor={Nick Early},
	pdfsubject={},
	pdfkeywords={},
	linkcolor=black, citecolor=black, urlcolor=blue, filecolor=blue,  pdfpagelayout=OneColumn, pdfnewwindow=true,  pdfstartview=XYZ, plainpages=false, pdfpagelabels}
\usepackage{breakurl}

\makeatletter

\theoremstyle{plain}
\newtheorem{thm}{\protect\theoremname}
\theoremstyle{plain}

\theoremstyle{plain}

\theoremstyle{remark}

\theoremstyle{plain}

\theoremstyle{plain}
\newtheorem{prop}[thm]{\protect\propositionname}
\theoremstyle{remark}

\theoremstyle{remark}
\newtheorem{rem}[thm]{\protect\remarkname}
\theoremstyle{definition}
\newtheorem{defn}[thm]{\protect\definitionname}
\theoremstyle{definition}
\newtheorem{example}[thm]{\protect\examplename}
\theoremstyle{plain}
\newtheorem{cor}[thm]{\protect\corollaryname}

\newcommand{\cellsize}{20}
\newlength{\cellsz} 
\newsavebox{\cell}
\sbox{\cell}{\begin{picture}(\cellsize,\cellsize)
	\put(0,0){\line(1,0){\cellsize}}
	\put(0,0){\line(0,1){\cellsize}}
	\put(\cellsize,0){\line(0,1){\cellsize}}
	\put(0,\cellsize){\line(1,0){\cellsize}}
	\end{picture}}
\newcommand\cellify[1]{\def\thearg{#1}\def\nothing{}%
	\ifx\thearg\nothing
	\vrule width0pt height\cellsz depth0pt\else
	\hbox to 0pt{\usebox{\cell} \hss}\fi%
	\vbox to \cellsz{
		\vss
		\hbox to \cellsz{\hss$#1$\hss}
		\vss}}
\newcommand\tableau[1]{\vtop{\let\\\cr
		\baselineskip -16000pt \lineskiplimit 16000pt \lineskip 0pt
		\ialign{&\cellify{##}\cr#1\crcr}}}
\newcommand{\kellsize}{20}
\newlength{\kellsz} 
\newsavebox{\kell}
\sbox{\kell}{\begin{picture}(\kellsize,\kellsize)
	\put(0,0){\line(1,0){\kellsize}}
	\put(0,0){\line(0,1){\kellsize}}
	\put(\kellsize,0){\line(0,1){\kellsize}}
	\put(0,\kellsize){\line(1,0){\kellsize}}
	\end{picture}}
\newcommand\kellify[1]{\def\thearg{#1}\def\nothing{}%
	\ifx\thearg\nothing
	\vrule width0pt height\kellsz depth0pt\else
	\hbox to 0pt{\usebox{\kell} \hss}\fi%
	\vbox to \kellsz{
		\vss
		\hbox to \kellsz{\hss$#1$\hss}
		\vss}}
\newcommand\ktableau[1]{\vtop{\let\\\cr
		\baselineskip -16000pt \lineskiplimit 16000pt \lineskip 0pt
		\ialign{&\kellify{##}\cr#1\crcr}}}

\newcommand{\sellsize}{10}
\newlength{\sellsz} 
\newsavebox{\sell}
\sbox{\sell}{\begin{picture}(\sellsize,8)
	\put(0,0){\line(1,0){\sellsize}}
	\put(0,0){\line(0,1){\sellsize}}
	\put(\sellsize,0){\line(0,1){\sellsize}}
	\put(0,\sellsize){\line(1,0){\sellsize}}
	\end{picture}}
\newcommand\sellify[1]{\def\thearg{#1}\def\nothing{}%
	\ifx\thearg\nothing
	\vrule width0pt height\sellsz depth0pt\else
	\hbox to 0pt{\usebox{\sell} \hss}\fi%
	\vbox to \sellsz{
		\vss
		\hbox to \sellsz{\hss$#1$\hss}
		\vss}}
\newcommand\stableau[1]{\vtop{\let\\\cr
		\baselineskip -16000pt \lineskiplimit 16000pt \lineskip 0pt
		\ialign{&\sellify{##}\cr#1\crcr}}}

\newcommand{\ssellsize}{5}
\newlength{\ssellsz} 
\newsavebox{\ssell}
\sbox{\ssell}{\begin{picture}(\ssellsize,4)
	\put(0,0){\line(1,0){\ssellsize}}
	\put(0,0){\line(0,1){\ssellsize}}
	\put(\ssellsize,0){\line(0,1){\ssellsize}}
	\put(0,\ssellsize){\line(1,0){\ssellsize}}
	\end{picture}}
\newcommand\ssellify[1]{\def\thearg{#1}\def\nothing{}%
	\ifx\thearg\nothing
	\vrule width0pt height\sellsz depth0pt\else
	\hbox to 0pt{\usebox{\ssell} \hss}\fi%
	\vbox to \ssellsz{
		\vss
		\hbox to \ssellsz{\hss$#1$\hss}
		\vss}}
\newcommand\sstableau[1]{\vtop{\let\\\cr
		\baselineskip -16000pt \lineskiplimit 16000pt \lineskip 0pt
		\ialign{&\ssellify{##}\cr#1\crcr}}}

\usepackage{ifpdf}
\ifpdf

\IfFileExists{lmodern.sty}
{\usepackage{lmodern}}{}

\fi 

\usepackage{multicol}

\makeatother

\providecommand{\claimname}{\inputencoding{latin9}Claim}
\providecommand{\conjecturename}{\inputencoding{latin9}Conjecture}
\providecommand{\corollaryname}{\inputencoding{latin9}Corollary}
\providecommand{\definitionname}{\inputencoding{latin9}Definition}
\providecommand{\examplename}{\inputencoding{latin9}Example}
\providecommand{\lemmaname}{\inputencoding{latin9}Lemma}
\providecommand{\notename}{\inputencoding{latin9}Note}
\providecommand{\propositionname}{\inputencoding{latin9}Proposition}
\providecommand{\questionname}{\inputencoding{latin9}Question}
\providecommand{\remarkname}{\inputencoding{latin9}Remark}
\providecommand{\theoremname}{\inputencoding{latin9}Theorem}

\newcommand{\symm}{\mathfrak{S}}

\newcommand\twoheaduparrow{\mathrel{\rotatebox{90}{$\twoheaduparrow$}}}
\newcommand\twoheaddownarrow{\mathrel{\rotatebox{270}{$\twoheaddownarrow$}}}

\begin{document}
	\author{Nick Early}
	\thanks{The author was partially supported by RTG grant NSF/DMS-1148634,\\
		Max Planck Institute for Mathematics in the Sciences, email: \href{mailto:earlnick@gmail.com}{earlnick@gmail.com}}
	
	\title[Generalized permutohedra in the kinematic space]{Generalized permutohedra in the kinematic space}
	
	\maketitle

	\begin{abstract}
	In this note, we study the permutohedral geometry of the singularities of a certain differential form introduced in recent work of Arkani-Hamed, Bai, He and Yan.  There it was observed that the poles of the form determine a family of polyhedra which have the same face lattice as that of the permutohedron.  We realize that family explicitly, proving that it in fact fills out the configuration space of a particularly well-behaved family of generalized permutohedra, the zonotopal generalized permutohedra, that are obtained as the Minkowski sums of line segments parallel to the root directions $e_i-e_j$.

	Finally we interpret Mizera's formula for the biadjoint scalar amplitude $m(\mathbb{I}_n,\mathbb{I}_n)$, restricted to a certain dimension $n-2$ subspace of the kinematic space, as a sum over the boundary components of the standard root cone, which is the conical hull of the roots $e_1-e_2,\ldots, e_{n-2}-e_{n-1}$.

\end{abstract}

\begingroup
\let\cleardoublepage\relax
\let\clearpage\relax
\tableofcontents
\endgroup

\section{The kinematic space}\label{sec: Kinematic space}

Let $I=\{a,1,2,\ldots,n-1,n,b \}$, where for our purposes $a,b$ are auxiliary indices.

\begin{defn}
	Define the kinematic space to be the subspace of $\mathbb{R}^{(n+2)^2}$,
	$$\mathcal{K}^n=\left\{(s_{ij})\in\mathbb{R}^{(n+2)^2}:s_{ii}=0,\ s_{ij} = s_{ji},\ i,j\in I,\text{ and } \sum_{j\in I:j\not=i} s_{ij}=0 \text{ for all }i\in I\right\}.$$

	Denote by $\mathcal{K}^n(\mathcal{D})$ the intersection in $\mathcal{K}^n$ of the $\binom{n}{2}$ affine hyperplanes $s_{ij} = -c_{ij}$ for all $1\le i<j\le n$ for given constants $\mathcal{D}=(c_{ij})$ with $c_{ij}\ge 0$.
\end{defn}
We denote 
$$s_{J} = \sum_{i,j\in J:\ i<j}s_{ij} $$
and  
$$s_{J_1\vert J_2} = \sum_{(i,j)\in J_1\times J_2} s_{ij}$$
for (nonempty) subsets $J, J_1,J_2$ of $I$, where $J_1\cap J_2=\emptyset$.  Let us adopt the natural convention that $s_{J}=0$ for any singlet $J=\{j\}$.

Note that $\dim(\mathcal{K}^n)=\binom{n+2}{2}-(n+2) = \frac{(n+2)(n-1)}{2}$, and
$$\dim(\mathcal{K}^n(\mathcal{D})) = \binom{n+2}{2} - (n+2)-\binom{n}{2} = n-1.$$

In \cite{worldsheet}, after adjusting the notation, the inequalities corresponding to the facets of a polyhedral cone were given as 
$$s_{aJ}=s_{aj_1\cdots j_k}\ge 0,$$
as $J$ ranges over all $2^n-2$ proper nonempty subsets of $\{1,\ldots, n\}$.

\begin{rem}
	The coordinates $s_{ij}$ actually have an essential structure which is not needed explicitly for our results.  They are called generalized \textit{Mandelstam} invariants.  They are constructed from momentum vectors $p_1,\ldots, p_{n+2}$ of a system of particles satisfying momentum conservation $\sum_{i=1}^{n+2} p_i=0$, in spacetime of dimension $D\ge n+1$ with the Minkowski inner product, and are defined by $s_{ij} = (p_i+p_j)^2 = 2p_i\cdot p_j$, due the assumption that particles are massless, that is $p_i\cdot p_i=0$, where we are using the notation $p^2:=p\cdot p$ for any linear combination of momentum vectors $p$.
\end{rem}

\begin{rem}
	It is worth pointing out that, as an $\symm_{n+2}$-module, $\mathcal{K}^n$ is irreducible, and is isomorphic to to $V_{(n,2)}$, that is the $\frac{(n+2)(n-1)}{2}$-dimensional irreducible representation labeled by the partition $(n,2)$ of $n+2$.  
\end{rem}

In polyhedral geometry, a zonotope is a Minkowski sum of line segments $\lbrack 0,v\rbrack$ for points $v\in \mathbb{R}^n$.  In particular, a zonotopal generalized permutohedron \cite{PostnikovPermutohedra}, see also \cite{FacesPermutohedra}, is the Minkowski sum of a collection of line segments parallel to the root directions $e_i-e_j$.  In Theorem \ref{thm:zonotopal permutohedron proof 3} we derive and use the formula for the characteristic function of the intersection of the set of affine hyperplanes $s_{ij} = -c_{ij}$ with the region determined by the set of $2^{n}-2$ inequalities 
$$s_{aJ}=s_{aj_1\cdots j_k}\ge 0,$$
where $J$ varies over all nonempty proper subsets of $\{1,\ldots, n\}$, to show in particular that each such intersection is a zonotopal generalized permutohedron, and all zonotopal generalized permutohedra are obtained in this way.

It may seem surprising a priori that the canonical form for the permutohedron, on the kinematic space, as in \cite{worldsheet}, would encode so specialized a family as the \textit{zonotopal} generalized permutohedra. However, it is also a much studied and indeed relatively well-understood special case which already has some powerful results and techniques that immediately become available.  For example, consider this: there is a certain quasi-symmetric function invariant of the normal cone of a generalized permutohedron, see Section 9.1 of \cite{Billera}.  It turns out for the zonotopal generalized permutohedra, this quasi-symmetric function coincides with Stanley's chromatic \textit{symmetric} function.  For other generalized permutohedra this invariant is likely to be only a quasisymmetric function; indeed, this happens already for the equilateral triangle.  

While it is not true for general graphs, it is currently being studied whether Stanley's chromatic symmetric function can distinguish those zonotopal generalized permutohedra which are encoded by trees.  This question was originally posed by Stanley in \cite{StanleyChromatic}.

\section{Main result}

Let us comment briefly on how we resolve and future-proof some potentially conflicting conventions.  In \cite{worldsheet}, $(n-3)$-dimensional generalized permutohedra were embedded in an ambient space of dimension $n-2$ in a system of $n$ particles.  However, in the usual mathematical literature on permutohedra, they are usually taken to be dimension $n-1$ in an ambient space of dimension $n$.  

For our main result we consider the configuration space of $(n-1)$-dimensional zonotopal generalized permutohedra in an $n$-dimensional ambient space; this increases the requisite number of particles to $n+2$.

However, in Sections \ref{sec: kinematic associahedron} and Appendix \ref{sec: Appendix triangulations dual associahedra} it is convenient to denote $m=n+1$.  

We first recall the formulation of the canonical form for the permutohedron from \cite{worldsheet}, which can be constructed from family of simplicial cones, called plates in \cite{EarlyCanonicalBasis}.  Plates are also related to the study of matroid subdivisions \cite{E2022}.

Let the matrix of constants $\mathcal{D}$ be given.

The canonical form for the permutohedron from Section 10 of \cite{worldsheet} is
\begin{eqnarray}\label{eqn: canonical form}
\left(\sum_{\sigma\in\symm_n}\frac{1}{\prod_{i=1}^{n-1}s_{a\sigma_1\sigma_2\cdots \sigma_i}}\right)d^{n-1}s,
\end{eqnarray}
where for compatibility with the conventions for generalized permutohedra our index set is taken to be $\mathcal{I} = \{a,1,\ldots, n,b\}$ rather than the set $\{1,\ldots, n\}$ from \cite{worldsheet}.  In Equation \eqref{eqn: canonical form}, the summand $1/\prod_{i=1}^{n-1}s_{a\sigma_1\sigma_2\cdots \sigma_i}$ is associated to the so-called multi-peripheral graph shown in Figure \ref{fig:multi-peripheralgraph}.
\begin{figure}[h!]
	\centering
	\includegraphics[width=0.4\linewidth]{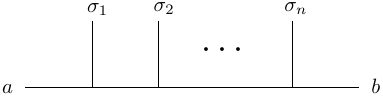}
	\caption{A multi-peripheral Feynman diagram, corresponding to the polyhedral cone defined by Equation \eqref{eqn:kinematicPlate}.  Here $\sigma=(\sigma_1,\ldots, \sigma_n)$ is any permutation of $\{1,\ldots, n\}$.}
	\label{fig:multi-peripheralgraph}
\end{figure}

Note that each pole $s_{a\sigma_1\cdots \sigma_i}=0$ which appears decomposes $\mathcal{K}^n$, and thus $\mathcal{K}^n(\mathcal{D})$, into two half spaces; choosing the half spaces $s_{a\sigma_1\cdots \sigma_i}\ge 0$, then for each of the $n!$ summands we obtain a flag of inequalities
\begin{eqnarray}\label{eqn:kinematicPlate}
\nonumber		s_{a\sigma_1} & \ge & 0\\
\nonumber		s_{a\sigma_1\sigma_2} & \ge & 0\\
s_{a\sigma_1\sigma_2\sigma_3}& \ge & 0\\
\nonumber		& \vdots & \\
\nonumber		s_{a\sigma_1\sigma_2\cdots \sigma_{n-1}} & \ge & 0\\
\nonumber		s_{a12\cdots n} & = & 0.
\end{eqnarray}

\begin{defn}
	Denote by $\lbrack\lbrack \sigma_1,\sigma_2,\ldots, \sigma_n\rbrack\rbrack_{\mathcal{D}}$ the characteristic function of the cone in $\mathcal{K}^n_{\mathcal{D}}$ determined by \eqref{eqn:kinematicPlate}.
	
	More generally, denote by $\lbrack\lbrack S_1,\ldots, S_k\rbrack\rbrack_\mathcal{D}$, where $(S_1,\ldots, S_k)$ is any ordered set partition of $\{1,\ldots, n\}$, the characteristic function of the cone in $\mathcal{K}^n(\mathcal{D})$ determined by the inequalities
	\begin{eqnarray}\label{eqn:kinematicPlateLumped}
	\nonumber	s_{aS_1} & \ge & 0\\
	\nonumber	s_{aS_1\cup S_2} &  \ge & 0\\
	s_{aS_1\cup S_2\cup S_3} &  \ge & 0\\
	\nonumber	& \vdots &\\
	\nonumber	s_{aS_1\cup S_2\cup\cdots\cup S_{k-1}} &  \ge & 0\\
	\nonumber	s_{a12\cdots n} & = & 0.
	\end{eqnarray}
\end{defn}

\begin{prop}\label{prop:linear relations}
	The equality in the last line of Equations \eqref{eqn:kinematicPlate} and \eqref{eqn:kinematicPlateLumped} holds identically on $\mathcal{K}^n$, and in particular on each affine subspace $\mathcal{K}^n(\mathcal{D})$.
\end{prop}
\begin{proof}
	Since 
	\begin{eqnarray*}
		0 & = &  \sum_{j\in I}\left(\sum_{i\not =j}s_{ij}\right) =2\sum_{i<j} s_{ij} = 2s_{a12\cdots nb} \\
		& = &  2\left( s_{a\vert 12\cdots n}+s_{12\cdots n}+s_{a 12\cdots n\vert b}\right)
	\end{eqnarray*}	
	and $s_{a 12\cdots n\vert b}=0$, we have automatically
	$$0=s_{a\vert 12\cdots n} + s_{12\cdots n}=s_{a12\cdots n}.$$
	
\end{proof}
Rearranging Equation \eqref{eqn:kinematicPlate} we obtain	
\begin{eqnarray*}
	s_{a\sigma_1} & \ge & 0\\
	s_{a\vert \sigma_1\sigma_2} & \ge & -s_{\sigma_1\sigma_2}\\
	s_{a\vert \sigma_1\sigma_2\sigma_3} & \ge & -(s_{\sigma_1\sigma_2}+s_{\sigma_1\sigma_2\vert \sigma_3})\\
	& \vdots &\\
	s_{a\vert 12\cdots n} & = & -(s_{12}+s_{12\vert 3}+\cdots+s_{12\cdots n-1\vert n}).
\end{eqnarray*}
Letting $x_i=s_{ai}$ for $i=1,\ldots, n$ and, following \cite{worldsheet}, putting $s_{ij}=-c_{ij}$ for the given constants $c_{ij}\in\mathcal{D}$, this becomes
\begin{eqnarray*}
	x_\sigma & \ge & 0\\
	x_{\sigma_1\sigma_2} & \ge & c_{12}\\
	x_{\sigma_1\sigma_2\sigma_3} & \ge & c_{\sigma_1\sigma_2}+c_{\sigma_1\sigma_2\vert \sigma_3}\\
	& \vdots &\\
	x_{12\cdots n} & = & c_{12\cdots n},
\end{eqnarray*}
which defines a permutohedral cone which we denote $\lbrack 1_{d_1} 2_{d_2}\ldots n_{d_n}\rbrack$, where $d_j = \sum_{i=1}^{j-1} c_{ij}$, that is, in the notation of \cite{EarlyCanonicalBasis}, the translation of the plate $\lbrack 1,2,\ldots, n\rbrack$ by the vector $(d_1,d_2,\ldots, d_n)$.

\begin{rem}
	The inequalities take on a pleasant form when we express the variables in terms of the momentum vectors $p_a,p_1,\ldots, p_n,p_b$.  Recall the standard notation $p^2:=p\cdot p$, where $p$ is any linear combination of momentum vectors $p_i$.  Then we have
\begin{eqnarray*}
	p_a\cdot p_{\sigma_1} & \ge & 0\\
	p_a\cdot (p_{\sigma_1}+p_{\sigma_2}) & \ge &  -p_{\sigma_1}\cdot p_{\sigma_2}\\
	p_a\cdot (p_{\sigma_1}+p_{\sigma_2}+p_{\sigma_3}) & \ge &-\sum_{1\le i<j\le 3}p_{\sigma_i}\cdot p_{\sigma_j}\\
	& \vdots &\\
	p_a\cdot (p_{\sigma_1}+p_{\sigma_2}+\cdots + p_{\sigma_{n}}) & = & -\sum_{1\le i<j\le n}p_{\sigma_i}\cdot p_{\sigma_j}\\
\end{eqnarray*}
\end{rem}
Recall that the matrix of constants $\mathcal{D}=(c_{ij})$ has been fixed.  Let us denote by $\lbrack\lbrack ij_{c_{ij}}\rbrack\rbrack$ the characteristic function of the interval $$\lbrack ij_{c_{ij}}\rbrack=\{c_{ij}(te_i+(1-t)e_j):0\le t\le 1\},$$  
where $e_1,\ldots, e_n$ is the standard basis for $\mathbb{R}^n$.

\begin{defn}
	Denote by $\mathcal{Z}_\mathcal{D}$ the \textit{zonotopal} generalized permutohedron \cite{PostnikovPermutohedra}, defined to be the Minkowski sum of the dilated root intervals $\lbrack ij_{c_{ij}}\rbrack$, which has characteristic function
	$$\lbrack\lbrack 12_{c_{12}}\rbrack + \lbrack 13_{c_{13}}\rbrack + \cdots + \lbrack 1n_{c_{1n}}\rbrack + \lbrack 23_{c_{23}}\rbrack + \cdots + \lbrack (n-1)n_{c_{(n-1)n}}\rbrack\rbrack.$$

\end{defn}
As a side remark, note that this may be equivalently expressed equivalently in the algebra of characteristic functions using the convolution product, with respect to the Euler characteristic, as
$$\lbrack\lbrack 12_{c_{12}}\rbrack\rbrack \bullet \lbrack\lbrack 13_{c_{13}}\rbrack\rbrack \bullet \cdots \bullet \lbrack\lbrack 1n_{c_{1n}}\rbrack\rbrack \bullet\lbrack\lbrack 23_{c_{23}}\rbrack\rbrack \bullet  \cdots \bullet \lbrack\lbrack (n-1)n_{c_{(n-1)n}}\rbrack\rbrack.$$

For details about the convolution product and related issues in convex geometry, see \cite{BarvinokPammersheim}.  See also \cite{EarlyCanonicalBasis}, which collected a subset of the basic results from convex geometry, in the notation used here.

\begin{thm}\label{thm:zonotopal permutohedron proof 3}
	The intersection $\Pi_\mathcal{D}$ in $\mathcal{K}^n(\mathcal{D})$ of the half regions $s_{aJ}\ge 0$, as $J$ varies over all $2^n-2$ proper nonempty subsets of $\{1,\ldots, n\}$, is the zonotopal generalized permutohedron $\mathcal{Z}_D$.
\end{thm}

\begin{proof}
	In the affine subspace $\mathcal{K}^n(\mathcal{D}) = \{(s)\in\mathcal{K}^n:s_{ij}=-c_{ij}, \text{ for } 1\le i<j\le n\}$, the inequalities $s_J\ge 0$ take the form $x_J\ge c_J$.  We claim that these equations define a generalized permutohedron.  Indeed, it follows from Theorem 6.3 of \cite{PostnikovPermutohedra}, see also \cite{AguiarAdrilaPermutohedra,Morton et al}, that the data $c_J$ as above determine a generalized permutohedron if and only if we have the supermodularity conditions
	$$c_I+c_J\le c_{I\cup J}+c_{I\cap J}$$
	for all nonempty subsets $I,J\subsetneq\{1,\ldots, n\}$.
	
	Set $A_{10}= I\setminus(I\cap J)$, $A_{11}=I\cap J$, and $A_{01}=J\setminus(I\cap J)$.  Then we have 
	\begin{eqnarray*}
		c_{I\cup J} = c_{A_{10}\sqcup A_{11}\sqcup A_{01}}
		& = & c_{A_{10}}+c_{A_{11}} + c_{A_{01}} + c_{A_{01}\vert  A_{11}} + c_{A_{01}\vert A_{10}} + c_{A_{10}\vert A_{11}}\\
		& = & (c_{A_{10}}+c_{A_{11}}+ c_{A_{10}\vert A_{11}}) + (c_{A_{01}}+c_{A_{11}} + c_{A_{01}\vert A_{11}}) + c_{A_{10}\vert A_{01}}  - c_{A_{11}}\\
		c_{I\cup J} + c_{I\cap J} & = & c_I+c_J +  c_{A_{10}\vert A_{01}}\\
		\Rightarrow c_I+c_J & = & c_{I\cup J} + c_{I\cap J} -  c_{A_{10}\vert A_{01}}\\
		& \le & c_{I\cup J} + c_{I\cap J},
	\end{eqnarray*}
	since  $c_{A_{10}\vert A_{01}}\ge 0$. 	Hence the equations $x_J\ge c_J$ define a generalized permutohedron which we denote by $\Pi_\mathcal{D}$.  It follows that the $n!$ vertices $v_\sigma$ are labeled by permutations $\sigma=(\sigma_1,\ldots,\sigma_n)$ of $(1,\ldots, n)$, and are obtained by solving the systems of equations 
	\begin{eqnarray*}
		s_{a\sigma_1} & = & 0\\
		s_{a\sigma_1\sigma_2} & = & 0\\
		s_{a\sigma_1\sigma_2\sigma_3} & = & 0\\
		&\vdots &\\
		s_{a\sigma\cdots\sigma_n} & = & 0,
	\end{eqnarray*}	
that is
	\begin{eqnarray*}
	x_{\sigma_1} & = & 0\\
	x_{\sigma_1\sigma_2} & = & c_{\sigma_1\sigma_2}\\
	x_{\sigma_1\sigma_2\sigma_3} & = & c_{\sigma_1\sigma_2\sigma_3}\\
	&\vdots &\\
	x_{a\sigma\cdots\sigma_n} & = & c_{\sigma_1\cdots\sigma_n}.
\end{eqnarray*}
The vertex $v_\sigma$ of $\Pi_\mathcal{D}$ is then given explicitly as
$$v_\sigma = c_{\sigma_1\sigma_2}e_{\sigma_2} + c_{\sigma_1\sigma_2\vert\sigma_3}e_{\sigma_3}+\cdots c_{\sigma_1\sigma_2\cdots\sigma_{n-1}\vert\sigma_n}e_{\sigma_n}.$$
We claim that $\mathcal{Z}_\mathcal{D}$ has the same vertex set as $\Pi_D$; this will show that the convex hulls coincide.

The Minkowski sum is given explicitly as the image of the linear map
$$x:\lbrack0,1\rbrack ^{\binom{n}{2}}\rightarrow \mathcal{K}^n(\mathcal{D})$$
defined by
$$\mathbf{t} \mapsto \sum_{i=1}^n x_i(\mathbf{t}) e_i=\sum_{1\le i<j\le n} c_{ij}(t_{ij}e_i+ (1-t_{ij})e_j)$$
keeping in mind the convention $t_{ij}=t_{ji}$ and $c_{ij}=c_{ji}$.  We remind that $e_1,\ldots, e_n$ is the standard basis for $\mathbb{R}^n$.  We claim that  $x(\mathbf{t})$ is in the permutohedron $\mathcal{K}^n(\mathcal{D})$ for all $0\le t_{ij} \le1$.

Collecting coefficients we have
\begin{eqnarray}\label{eqn:parametrizationZonotope}
x_1(\mathbf{t}) & = & (c_{12}t_{12}+c_{13}t_{13}+\cdots+c_{1n}t_{1n})\nonumber\\
x_2(\mathbf{t}) & = & (c_{21}(1-t_{21})+c_{23}t_{23}+\cdots+c_{2n}t_{2n})\nonumber\\
x_3(\mathbf{t}) & = & (c_{31}(1-t_{31})+c_{32}(1-t_{32})+c_{34}t_{34}+\cdots+c_{3n}t_{3n})\\
& \vdots & \nonumber\\
x_n(\mathbf{t}) & = & (c_{n1}(1-t_{n1})+c_{n2}(1-t_{n2})+\cdots+c_{n(n-1)}(1-t_{n(n-1)})).\nonumber
\end{eqnarray}
Then
\begin{eqnarray*}
	x_J & = & \sum_{j\in J}\left(\sum_{i<j} (c_{ij}t_{ij}) + \sum_{i>j} (c_{ij}(1-t_{ij}))\right)\\
	& = & c_J + \sum_{j\in J}\left(\sum_{i<j;\ i\not\in J}c_{ij} t_{ij} + \sum_{i>j;\ i\not\in J}c_{ij}(1- t_{ij})\right)\\
	&\ge & c_{J},
\end{eqnarray*}
as $0\le t_{ij} \le 1$ for all $1\le i<j\le n$.  This shows that $\mathcal{Z}_\mathcal{D} \subseteq \Pi_\mathcal{D}$.  

For the inclusion $\mathcal{Z}_\mathcal{D} \supseteq \Pi_\mathcal{D}$, since both $\mathcal{Z}_\mathcal{D}$ and $\Pi_D$ are (convex) generalized permutohedra and $\Pi_D$ is the convex hull of its vertices, it suffices to check that every vertex of $\mathcal{Z}_\mathcal{D}$ is a vertex of $\Pi_\mathcal{D}$.

Let a permutation $\sigma=(\sigma_1,\ldots, \sigma_n)$ be given. Denote by
$$\mathcal{I}_\sigma=\{(i,j):\sigma_i>\sigma_j,\  1\le i<j\le n\}$$
the set of inversions of $\sigma$.  Set 
$$t_{ij}= 1,\text{ if } (i,j)\in\mathcal{I}_\sigma$$
and
$$t_{ij}= 0,\text{ if } (i,j)\not\in\mathcal{I}_\sigma.$$
Collect these values in a vector $\mathbf{t}$.  Then from Equation \eqref{eqn:parametrizationZonotope} it follows that 
$$x(\mathbf{t})=\sum_{j=1}^n \left(\sum_{i<j}c_{\sigma_i\sigma_j}\right)e_{\sigma_j}=c_{\sigma_1\sigma_2}e_{\sigma_2} + c_{\sigma_1\sigma_2\vert\sigma_3}e_{\sigma_3}+\cdots c_{\sigma_1\sigma_2\cdots\sigma_{n-1}\vert\sigma_n}e_{\sigma_n}$$
which is also the expression for $v_\sigma\in\Pi_\mathcal{D}$.
\end{proof}
\begin{cor}
	 The characteristic function of $\mathcal{Z}_\mathcal{D}$ equals the alternating sum:
		\begin{eqnarray}\label{eqn:alternating sum zonotopal}
		\lbrack\mathcal{Z}_\mathcal{D}\rbrack = \sum_{\mathbf{T}}(-1)^{n-\text{len}(\mathbf{T})}\lbrack\lbrack \mathbf{T}\rbrack\rbrack_\mathcal{D},
		\end{eqnarray}
		where $\mathbf{T}$ varies over all ordered set partitions of $\{1,\ldots, n\}$.
\end{cor}

\begin{rem}
	Modulo characteristic functions of tangent cones to faces of dimension $\ge 1$, labeled by ordered set partitions $(S_1,\ldots, S_k)$ where at least one block is not a singlet, the expression in Equation \eqref{eqn:alternating sum zonotopal} of Theorem \ref{thm:zonotopal permutohedron proof 3} is a sum of $n!$ characteristic functions of tangent cones to vertices of $\Pi_\mathcal{D}$ all having the same sign $+1$, in alignment with the formula in Equation \eqref{eqn: canonical form} above, for the canonical form from Section 10 of \cite{worldsheet}.
\end{rem}

\begin{example}\label{example: zonotopal Permutohedron 1}
	Set $c_{i,j}=1$ for all $1\le i<j\le n$, so $d_i=i-1$.  Then the resulting plate $\lbrack 1_0 2_1 3_2,\ldots, n_{n-1}\rbrack$ is the permutohedral cone which is \textit{tangent} at the vertex $(0,1,2,\ldots, n-1)$ of the usual permutohedron obtained as the convex hull of permutations of $(0,1,2,\ldots, n-1)$:
	\begin{eqnarray*}
		x_1 & \ge & 0\\
		x_{12} & \ge & 1=\binom{2}{2}\\
		x_{123} & \ge & 1+2=\binom{3}{2}\\
		& \vdots &\\
		x_{12\cdots n} & = & 1+2+\cdots+(n-1)=\binom{n}{2}.
	\end{eqnarray*}
\end{example}

\begin{example}\label{example: zonotopal Permutohedron 2}
	In the limiting case $c_{i,j}\rightarrow 0$ for all $1\le i<j\le n$, then via the identification of $s_{ai}$ with $x_i$ we have 
	\begin{eqnarray*}
		x_1 & \ge & 0\\
		x_{12} & \ge & 0\\
		x_{123} & \ge & 0\\
		& \vdots &\\
		x_{12\cdots n} & = & 0,
	\end{eqnarray*}
	which determine the plate $\lbrack 1,2,\ldots, n\rbrack$, and we are in the setting of \cite{EarlyCanonicalBasis}.
\end{example}

	\begin{figure}[h!]
	\centering
	\includegraphics[width=0.55\linewidth]{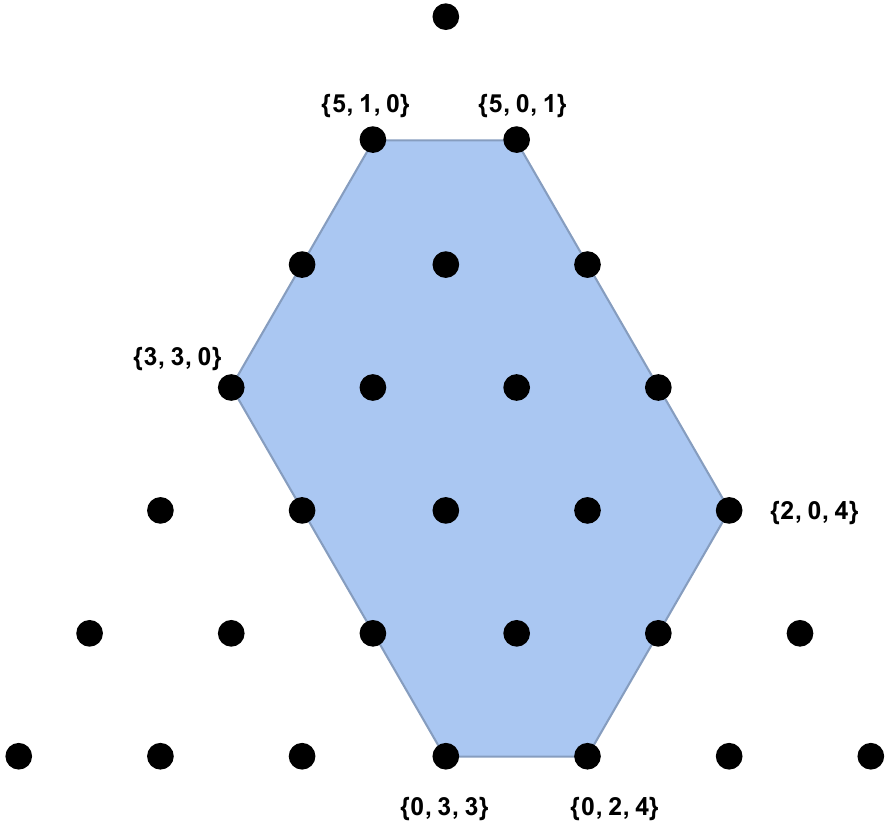}
	\caption{Zonotopal generalized permutohedron for Example \ref{example: zonotopal generalized permutohedron}, $(c_{12},c_{23},c_{13})=(2,1,3)$.  All edges parallel to $e_i-e_j$ have length $c_{ij}$.}
	\label{fig:zonotopalpermutohedron3coordinates}
	\end{figure}
	\begin{example}\label{example: zonotopal generalized permutohedron}
		Consider the case $\mathcal{K}^3$.  Let the matrix of constants $\mathcal{D}$ be given, with nonzero entries $c_{12}=c_{21}=2,\ c_{23}=c_{32}=1,\ c_{13}=c_{31}=3$.  The characteristic function of the zonotopal generalized permutohedron in Figure \ref{fig:zonotopalpermutohedron3coordinates} has the following expansion:
		$$\lbrack\lbrack 12_2\rbrack\rbrack \bullet \lbrack\lbrack 23_1\rbrack\rbrack\bullet \lbrack\lbrack 13_3\rbrack\rbrack $$
		$$=\lbrack\lbrack 1,2,3\rbrack\rbrack_{\mathcal{D}} + \lbrack\lbrack 2,1,3\rbrack\rbrack_{\mathcal{D}} + \lbrack\lbrack 2,3,1\rbrack\rbrack_{\mathcal{D}} + \lbrack\lbrack 3,2,1\rbrack\rbrack_{\mathcal{D}} + \lbrack\lbrack 3,1,2\rbrack\rbrack_{\mathcal{D}} + \lbrack\lbrack 1,3,2\rbrack\rbrack_{\mathcal{D}}$$
		$$-\left(\lbrack\lbrack 1,23\rbrack\rbrack_{\mathcal{D}} +\lbrack\lbrack 12,3\rbrack\rbrack_{\mathcal{D}} + \lbrack\lbrack 2,13\rbrack\rbrack_{\mathcal{D}} +\lbrack\lbrack 23,1\rbrack\rbrack_{\mathcal{D}} +\lbrack\lbrack 3,12\rbrack\rbrack_{\mathcal{D}} +\lbrack\lbrack 13,2\rbrack\rbrack_{\mathcal{D}}   \right)$$
		$$+\lbrack\lbrack 123\rbrack\rbrack_\mathcal{D}$$
		$$=\lbrack\lbrack 1_0 2_2 3_4\rbrack\rbrack + \lbrack\lbrack 2_0 1_2 3_4\rbrack\rbrack + \lbrack\lbrack 2_0 3_1 1_5\rbrack\rbrack + \lbrack\lbrack 3_0 2_1 1_5\rbrack\rbrack +\lbrack\lbrack 3_0 1_3 2_3\rbrack\rbrack + \lbrack\lbrack 1_0 3_3 2_3\rbrack\rbrack$$
		$$-\left(\lbrack\lbrack 1_0 23_6\rbrack\rbrack + \lbrack\lbrack 12_2 3_4\rbrack\rbrack +\lbrack\lbrack 2_0 13_6\rbrack\rbrack +\lbrack\lbrack 23_1 1_5\rbrack\rbrack + \lbrack\lbrack 3_0 12_6\rbrack\rbrack + \lbrack\lbrack 13_3 2_3\rbrack\rbrack \right)$$
		$$+\lbrack\lbrack 123_6\rbrack\rbrack,$$
	where for example $\lbrack\lbrack 1_0 2_2 3_4\rbrack\rbrack$ is the characteristic function of the cone at the vertex $(0,2,4)$ opening toward the upper left, cut out by the inequalities
	\begin{eqnarray*}
		x_1 & \ge & 0\\
		x_{12} & \ge & 0+2\\
		x_{123} & = & 0+2+4.
	\end{eqnarray*}
	The corresponding sum of fractions from \eqref{eqn: canonical form} in terms of generalized Mandelstam variables $s_J$ is
	$$\frac{1}{s_{a1}s_{a12}} + \frac{1}{s_{a1}s_{a13}} + \frac{1}{s_{a2}s_{a21}} + \frac{1}{s_{a2}s_{a23}} + \frac{1}{s_{a3}s_{a31}} + \frac{1}{s_{a3}s_{a32}}$$
\begin{small}
$$=\frac{1}{(x_1)(x_{12}-2)} + \frac{1}{(x_1)(x_{13}-3)} + \frac{1}{(x_2)(x_{21}-2)} + \frac{1}{(x_2)(x_{23}-1)} + \frac{1}{(x_3)(x_{31}-3)} + \frac{1}{(x_3)(x_{32}-1)},$$
\end{small}	
where we recall that $x_{ij} = x_i+x_j$.

\end{example}

\begin{example}
See Figure \ref{fig:pointexplodedhexagon1} for the explosion of a point to a hexagon, as the Minkowski sums $\lbrack 12_\varepsilon\rbrack + \lbrack 13_\varepsilon\rbrack+\lbrack 23_\varepsilon\rbrack = \left\{u_{12}+u_{13}+u_{23}:u_{ij}\in\lbrack ij_\varepsilon\rbrack\right\}$ from Theorem \ref{thm:zonotopal permutohedron proof 3} for (small) $\varepsilon\ge 0$.
\begin{figure}[h!]
	\centering
	\includegraphics[width=0.4\linewidth]{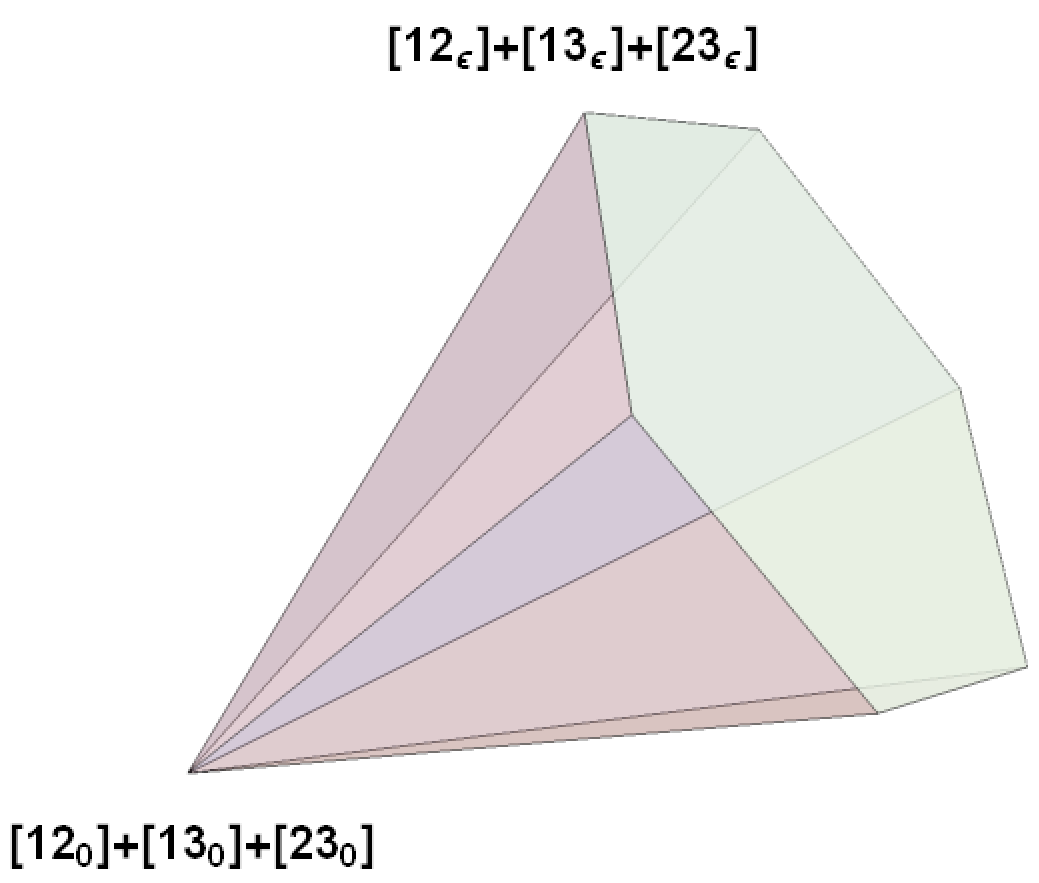}
	\caption{Exploding a point to a hexagon: $(c_{12},c_{13},c_{23})=(\varepsilon,\varepsilon,\varepsilon)$ for $\varepsilon\ge 0$}
	\label{fig:pointexplodedhexagon1}
\end{figure}
When $\varepsilon>0$ the expression does not have a canonical simplification,
\begin{small}
	$$\frac{1}{x_2 \left(x_1+x_2-\epsilon \right)}+\frac{1}{x_1 \left(x_1+x_3-\epsilon \right)}+\frac{1}{x_3 \left(x_1+x_3-\epsilon \right)}+\frac{1}{x_2 \left(x_2+x_3-\epsilon \right)}+\frac{1}{x_3 \left(x_2+x_3-\epsilon \right)}+\frac{1}{x_1 \left(x_1+x_2-\epsilon \right)},$$
but in the limit $\varepsilon\rightarrow0$ we have
\begin{small}
	$$\frac{1}{x_2 \left(x_1+x_2\right)}+\frac{1}{x_1 \left(x_1+x_3\right)}+\frac{1}{x_3 \left(x_1+x_3\right)}+\frac{1}{x_2 \left(x_2+x_3\right)}+\frac{1}{x_3 \left(x_2+x_3\right)}+\frac{1}{x_1 \left(x_1+x_2\right)}=\frac{x_1+x_2+x_3}{x_1 x_2 x_3}=0,$$
\end{small}
since the numerator vanishes identically.  This can also be seen from general theory, as higher codimension cones (in particular, here, the point at the origin) are in the kernel of the valuation induced by the integral Laplace transform.
\end{small}

\end{example}

\section{Generalized associahedra in the kinematic space}\label{sec: kinematic associahedron}

The standard realization of the associahedron was given, see \cite{Loday}, in terms of facet inequalities, as
$$\left\{x\in\mathbb{R}^n:\ \sum_{i=1}^n x_i=\binom{n}{2},\ \sum_{i\in I}x_i \ge \binom{\vert I\vert +1}{2}: I=\lbrack a,b\rbrack\subsetneq \{1,\ldots, n\} \right\},$$
where $\lbrack a,b\rbrack$ runs over all proper subintervals of $\{1,\ldots, m\}$ with $b-a\ge 2$.

Having a binomial coefficients for the vertex coordinates begs the question whether they are counting some quantity.  We show in what follows that they are counting the Mandelstam parameters $s_{ij}$ for all $i\not\in\{j-1,j,j+1\}$, specialized to the hyperplane where $s_{ij}= -c_{ij}=-1$.

The following construction in the kinematic for an arbitrary graph appeared first in \cite{HeSong2}, where the resulting generalized permutohedron is called a \textit{Cayley} polytope.

Let a collection of constants 
$$\mathcal{D}=\{c_{ij}=c_{ji}\ge 0:1\le i,j\le n+1,\ i\not\in\{j-1,j,j+1\}\}$$
be given.  Define 
$$\mathcal{A}^{n}(\mathcal{D})=\left\{(s)\in\mathcal{K}^{n}:\ s_{ij}=-c_{ij}\text{ whenever } i\not\in\{j-1,j,j+1\};\  s_{\lbrack a,b\rbrack }\ge 0,\ i,j,a,b=1,\ldots, n+1 \right\},$$
where we use the notation $\lbrack a,b\rbrack=\{a,a+1,\ldots, b\}$.
\begin{prop}
The set $\mathcal{A}^n(\mathcal{D})$ is a generalized permutohedron, and can be expressed as a Minkowski sum, as 
$$\{\lbrack (i\ i+1\cdots j-1\ j)_{c_{i,j+1}}\rbrack: 1\le i<j\le n\},$$
where $\lbrack (i\ i+1\cdots j-1\ j)_{c_{i,j+1}}\rbrack$ is the dilated simplex which is given by the convex hull of the set of vertices $\{c_{i,j+1}e_i,c_{i,j+1}e_{i+1},\ldots,c_{i,j+1}e_{j-1}, c_{i,j+1}e_j\}$.  In the case when all constants $c_{i,j+1}$ tend to $1$, then we recover exactly the usual associahedron.
\end{prop}
\begin{proof}[Sketch of proof]
	This follows from Proposition 7.5 in \cite{PostnikovPermutohedra}, where the dilation parameter $y_J$ for each contiguous subset $\{i,i+1,\ldots, j\}$ is now $c_{i,j+1}$, and where $y_J=0$ for all non-contiguous subsets $J\subset\{1,\ldots, m\}$.
\end{proof}

It is easy to check verify from the definition of $\mathcal{A}^n$, using the property $s_{I} = s_{I^c}$ for $I\subsetneq$ a nonempty subset of $\{1,\ldots, n+2\}$, that the kinematic space has a natural action of $\mathbb{Z}\slash (n+2)\mathbb{Z}$ which preserves the set of equations which cut out the $(n-1)$-dimensional associahedron. 

Indeed, then the pointwise action on the associahedron by $\sigma=(12\cdots n)$ 
becomes
\begin{eqnarray*}
	& & (s_{12},s_{23},\ldots, s_{n-2,n-1})\mapsto (s_{23}, s_{34},\ldots, s_{n-1,n}) = (s_{23}, s_{34},\ldots, s_{n-2,n-1},s_{12\cdots n-2})\\
	& = & \left(s_{23},s_{34},\ldots, s_{n-2,n-1}, \sum_{i=1}^{n-2}s_{i,i+1} - \sum_{1\le i<j-1\le n-3} c_{ij}\right),
\end{eqnarray*}

where the sum is over the constants $s_{i,j} = -c_{i,j}$ having nonadjacent indices $i<j-1\le n-3$.

This gives rise to Proposition \ref{prop:cyclic action associahedron}.

\begin{prop}\label{prop:cyclic action associahedron}
	The action of the $(n+2)$-cycle $\sigma=(12\cdots (n+2))$ on the kinematic space $\mathcal{K}^n$ preserves the natural embedding of the $(n-1)$-dimensional associahedron.
\end{prop}

\begin{example}
	In the case $m=5$ we have variables $(x_1,x_2,x_3) = (s_{12},s_{23},s_{34})$, where $x_{1}+x_2+x_3$ will be constant.  In terms of the Mandelstam variables, the 2-dimensional associahedron is cut out by the inequalities
	$$s_{12},s_{23}, s_{34},s_{123},s_{234} \ge 0.$$
	The inequalities defining the tangent cones at the five vertices are labeled by nesting intervals, as
	$$(s_{12},s_{123}\ge 0,\ \ s_{23},s_{123}\ge 0,\ \ s_{23},s_{234}\ge 0,\ \ s_{34},s_{234}\ge 0,\ \ s_{12},s_{34}\ge 0).$$
	These map termwise under $\sigma=(12345)$ to
	$$(s_{23},s_{234}\ge 0,\ \ s_{34},s_{234}\ge 0,\ \ s_{34},s_{12}\ge 0,\ \ s_{123},s_{12}\ge 0,\ \ s_{23},s_{123}\ge 0),$$
	having used the relations $s_{345}=s_{12}$ and $s_{45} = s_{123}$ to remove the index $5$.
	Then, in the $x$-coordinates this becomes
	\begin{eqnarray*}
		x_{1} & \ge & 0\\
		x_2 & \ge & 0\\
		x_3 & \ge & 0\\
		x_{1}+x_2 & \ge & c_{13}\\
		x_{2}+x_3 & \ge & c_{24}\\
		x_{1}+x_2+x_3 & = & c_{13}+c_{24} +c_{14},
	\end{eqnarray*}
	where the last line holds identically in the kinematic space.  The vertices are given by 
	$$\{(0,c_{13},c_{24}+c_{14}),(0,c_{13}+c_{24} + c_{14},0),(c_{13},0,c_{24}+c_{14}),(c_{13}+c_{14},0,c_{24}),(c_{13}+c_{14},c_{24},0)\}.$$
	This is the Minkowski sum of a triangle of edge length 3, and two line segments.  In terms of characteristic functions, using the convolution product we have 
	$$\lbrack \lbrack 123_{c_{14}}\rbrack\rbrack \bullet  \lbrack \lbrack 12_{c_{13}}\rbrack \rbrack \bullet \lbrack\lbrack 23_{c_{24}}\rbrack,$$
	where the triangle and two lines are given by 
	$$\lbrack \lbrack 123_{c_{14}}\rbrack\rbrack = \left\{x\in \lbrack 0,c_{14}\rbrack^3: \sum_{i=1}^3x_i=c_{14}\right\},$$
	$$\lbrack \lbrack 12_{c_{13}}\rbrack\rbrack = \left\{x\in \lbrack 0,c_{13}\rbrack^3: \sum_{i=1}^3x_i=c_{13},\  x_3=0\right\},$$
	$$\lbrack \lbrack 23_{c_{24}}\rbrack\rbrack = \left\{x\in \lbrack 0,c_{24}\rbrack^3: \sum_{i=1}^3x_i=c_{24},\  x_1=0\right\}.$$
	
	In the case that $(c_{13},c_{24},c_{14}) = (1,1,1)$ we recover the usual associahedron, as a generalized permutohedron, see Figure \ref{fig:associahedron}.
	\begin{figure}[h!]
		\centering
		\includegraphics[width=0.45\linewidth]{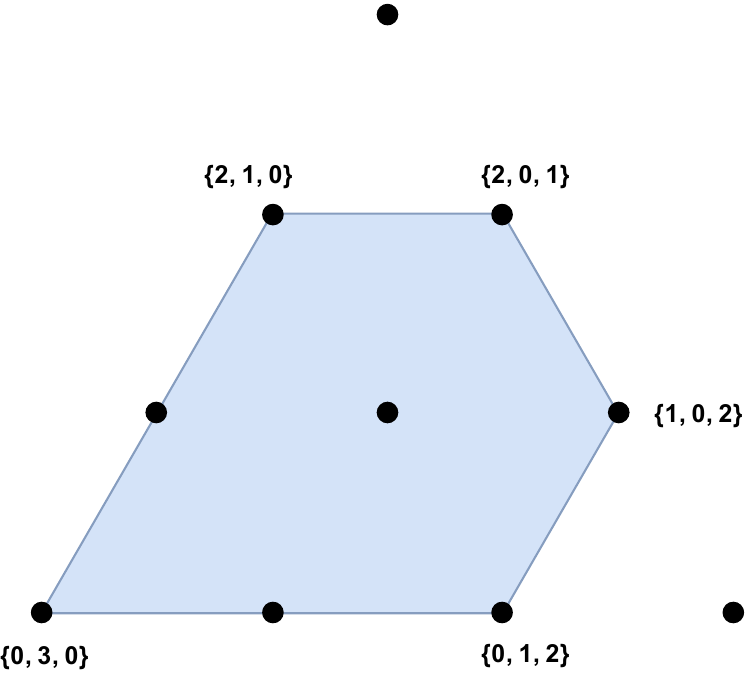}
		\caption{}
		\label{fig:associahedron}
	\end{figure}
	
	For the case that $(c_{13},c_{24},c_{14}) = (2,1,3)$ see see Figure \ref{fig:associahedron2}.
	\begin{figure}[h!]
		\centering
		\includegraphics[width=0.45\linewidth]{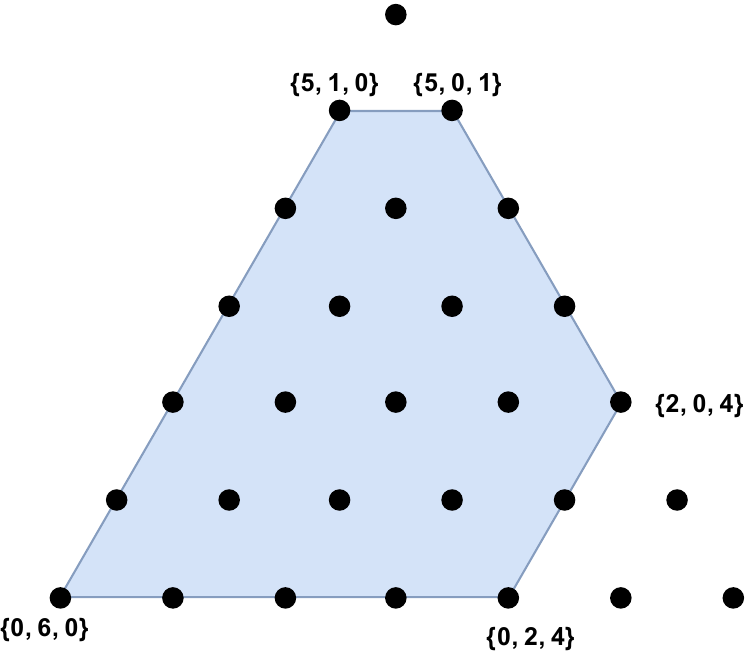}
		\caption{}
		\label{fig:associahedron2}
	\end{figure}
\end{example}

\newpage
\section{Triangulations of permutohedral cones and the associahedron}\label{sec: Appendix triangulations dual associahedra}
Let us fix $m=n+1$.

We here illustrate the \textit{tree} triangulation of the plate $\lbrack 1,2,\ldots, m\rbrack$, which we recall from \cite{EarlyCanonicalBasis} can be expressed as the conical hull
$$\lbrack 1,2,\ldots, m\rbrack = \langle e_1-e_2,\ldots, e_{m-1}-e_{m}\rangle_+ := \left\{t_1(e_1-e_2)+\cdots +t_{m-1}(e_{m-1}-e_{m}):t_i\ge 0\right\}.$$

Fix an order $(1,\ldots, m)$ for $m$ vertices on the line.  Following \cite{GelfandGraevPostnikov}, we see that the set of (unlayered) binary trees growing up from a given root, having $1,\ldots, m$ as leaves, are in an obvious bijection with directed trees with edges 
$$\left\{(i_1,j_1),\ldots, (i_{m-2},j_{m-2}),(1,m)\right\},$$
such that the intervals $\{i_a,i_{a}+1,\ldots, j_a\}$ are either nested or disjoint, and where $(1,m)$ is always an edge.  See Figure \ref{fig:triangulation-standard-plate-dual-associahedron-2} for the case $m=4$.
\begin{defn}
Denote by $\text{Trees}^{m}$ the set of all such trees with leaves in the order $(1,\ldots, m)$.  Call a \textit{partial tree} a directed graph which is obtained by removing a subset of the edges 
$$\{(i_1,j_1),\ldots, (i_{m-2},j_{m-2})\}$$
from a tree $\{(i_1,j_1),\ldots, (i_{m-2},j_{m-2}),(1,m)\}\in\text{Trees}^{m}$.  Denote by $\text{Tree}_{\le}^n$ the set of all partial trees.
\end{defn}

\begin{prop}
	The set of simplicial cones 
	$$\left\{\langle(i_1,j_1),\ldots, (i_{m-2},j_{m-2}),(1,m)\rangle_+: \{(i_1,j_1),\ldots, (i_{m-2},j_{m-2}),(1,m)\}\in\text{Trees}^{m} \right\}$$
	decomposes $\langle e_1-e_2,\ldots, e_{m-1}-e_{m}\rangle_+$ into $C_{m-1}=\frac{1}{m}\binom{2(m-1)}{m-1}$ simplicial cones which have disjoint interiors, where $C_a$ is the $a^\text{th}$ Catalan number.
\end{prop}
\begin{proof}
This follows by a slight extension of Theorem 6.3 in \cite{GelfandGraevPostnikov}, replacing the \textit{convex} hull of the positive roots with their \textit{conical} hull 
$$\langle \{e_i-e_j:1\le i<j\le m\}\rangle_+=\langle e_1-e_2,\ldots, e_{m-1}-e_{m}\rangle_+ = \lbrack 1,2,\ldots, m\rbrack.$$
\end{proof}

\begin{example}\label{example: triangulation n=3}
	With $m=4$ we have the triangulation
	$$\langle e_1-e_2,e_2-e_3\rangle_+ = \langle e_1-e_3,e_2-e_3\rangle_+\cup \langle e_1-e_2,e_1-e_3\rangle_+.$$
	Note that this provides a nice interpretation of the fundamental rational function identity
	$$\frac{1}{(y_1-y_2)(y_2-y_3)} = \frac{1}{(y_1-y_3)(y_2-y_3)} + \frac{1}{(y_1-y_2)(y_1-y_3)},$$
	where the common boundary line has been ignored.  This could be interpreted as an identity among canonical forms of positive geometries.
\end{example}

Then, there is a natural duality between the triangulated plate $\lbrack 1,2,\ldots, m\rbrack$ and the $(m-2)$-dimensional associahedron. 

\begin{prop}
	The face poset of the set of simplicial cones in the tree triangulation of the standard plate $\pi_0=\langle e_1-e_2,\ldots, e_{m-1}-e_m\rangle_+$ is in duality with the set of tangent cones to the associahedron.
\end{prop}

\begin{proof}
	Let us fix $s_{ij} = -c_{ij}=-1$ for all nonadjacent indices $i<j-1$, for $1\le i,j\le m$.  Collect these values of the $c_{ij}$ in the matrix $\mathcal{D}$, as usual.
	
	To the face given by the conical hull $\langle e_{i_1}-e_{j_1},\ldots, e_{i_a}-e_{j_a}\rangle_+$ of the tree triangulation of $\pi_0$ encoded by the partial tree $\{(i_1,j_1),\ldots, (i_a,j_a)\}\in\text{Tree}_{\le}^m$, where $1\le a\le m-1,$ we assign the tangent cone to the face of the associahedron:
	$$\langle e_{i_1}-e_{j_1},\ldots, e_{i_a}-e_{j_a}\rangle_+\mapsto \{s\in\mathcal{K}^m(\mathcal{D}): s_{i_1\cdots j_1+1}\ge 0,\ldots, s_{i_a\cdots j_a+1}\ge 0 \}.$$
	For the converse, note that the elements $\{(i_1,j_1),\ldots, (i_{m-1}-j_{m-1})\}\in \text{Tree}^m$ label the (top-dimensional) simplicial cones in the triangulation, and that these are in duality with exactly the tangent cones to the vertices of the associahedron, according to the correspondence above.  It is easy to see from the construction that this correspondence reverses inclusion of sets.  Indeed, any face of the triangulation is obtained by removing 1 or more generators $\{e_{i_\alpha}-e_{j_\alpha}:\alpha\in A\}$ from some simplicial cone, and conversely the tangent cone to any face of the associahedron is obtained by removing 1 or more inequalities 
	$$\{s_{i_\alpha\cdots s_{j_\alpha}}\ge 0:\alpha\in A\}.$$
	 This completes the proof.
	
\end{proof}

Recall the following formula for the Laplace transform: if $\pi = \langle e_{i_1}-e_{j_1},\ldots, e_{i_{m-1}}-e_{j_{m}}\rangle_+$ is a (simplicial) cone, then the Laplace transform of $\pi$ is given by
\begin{eqnarray}\label{eq: integral Laplace transform}
	\int_{u\in \pi} e^{-u\cdot \mathbf{y}}du=\int_{t_i\ge 0}e^{-t_1(y_{i_1}-y_{j_1})-\cdots-t_{m-1}(y_{i_{m-1}}-y_{j_{m-1}})} =\frac{1}{(y_{i_1}-y_{j_1})\cdots (y_{i_{m-1}}-y_{j_{m-1}})}.
\end{eqnarray}
Such Laplace transformations have appeared recently in context of positive geometry, see \cite[Equation 7.186]{Positivegeometry}.

Returning to Example \ref{example: triangulation n=3}, we remark that, by the Laplace transform in Equation \eqref{eq: integral Laplace transform}, the triangulation of cones
$$\langle e_1-e_2,e_2-e_3\rangle_+ = \langle e_1-e_3,e_2-e_3\rangle_+\cup \langle e_1-e_2,e_1-e_3\rangle_+.$$
provides a nice interpretation of the fundamental rational function identity
$$\frac{1}{(y_1-y_2)(y_2-y_3)} = \frac{1}{(y_1-y_3)(y_2-y_3)} + \frac{1}{(y_1-y_2)(y_1-y_3)},$$
where the common boundary line has been ignored.  Identities of this form can be viewed as identities between positive geometries \cite{Positivegeometry}.  
\begin{figure}[h!]
	\centering
	\includegraphics[width=1\linewidth]{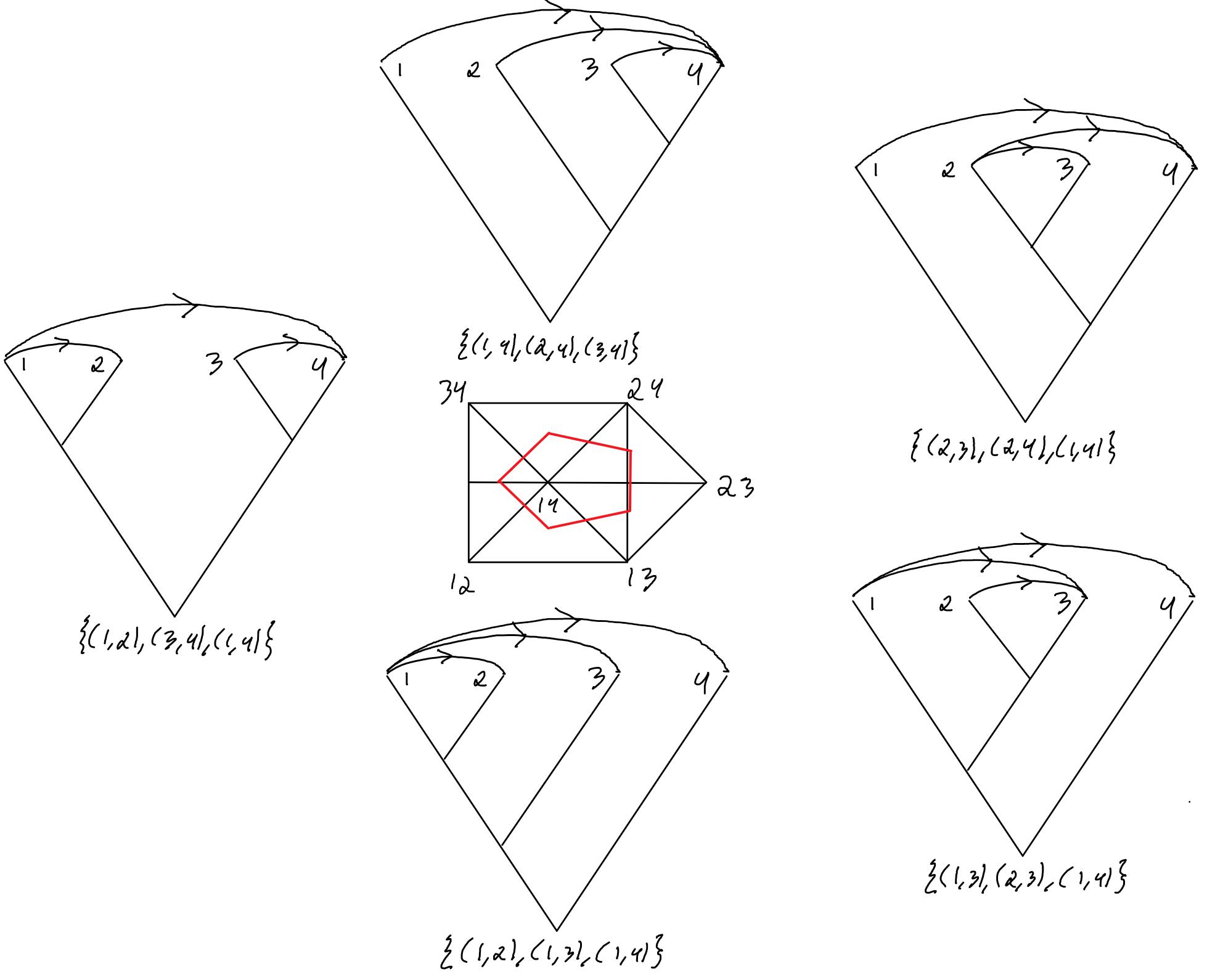}
	\caption{Top-dimensional simplicial cones in the triangulation of the standard plate $\lbrack 1,2,3,4\rbrack$ are dual to tangent cones to vertices of the associahedron (outlined schematically in red).  See Example \ref{Example:TriangulationStandardPlate4Coorsdinates}.}
	\label{fig:triangulation-standard-plate-dual-associahedron-2}
\end{figure}

\begin{figure}[h!]
	\centering
	\includegraphics[width=0.35\linewidth]{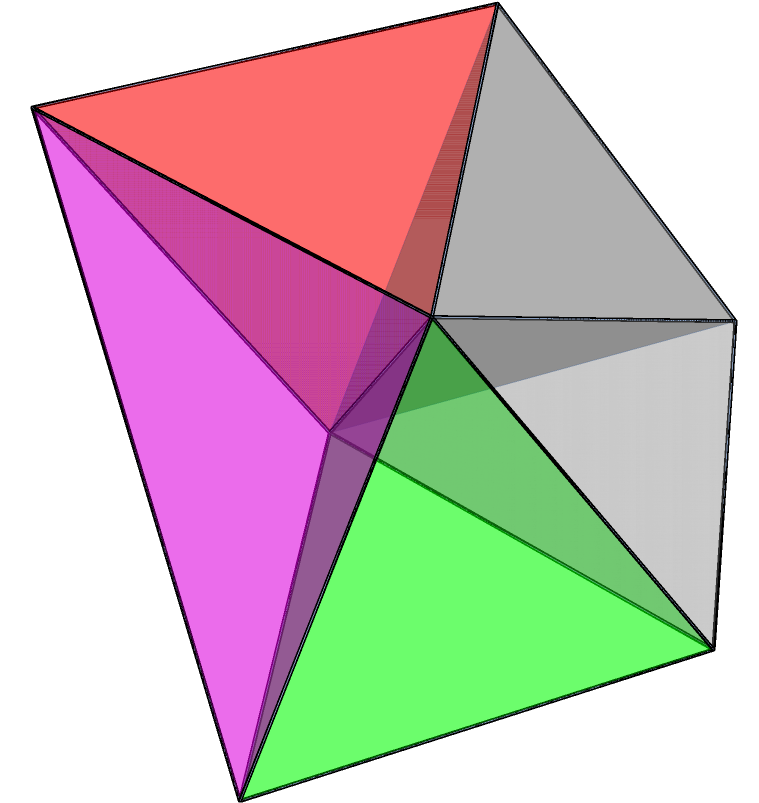}
	\caption{See Example \ref{Example:TriangulationStandardPlate4Coorsdinates}; triangulation of the root cone $\langle e_1-e_2,e_2-e_3,e_3-e_4\rangle_+.$}
	\label{fig:triangulationstandardplate}
\end{figure}

\begin{example}\label{Example:TriangulationStandardPlate4Coorsdinates}
	Set $h_{ij} = e_i-e_j$.  We consider the case $m=5$.
	
	The triangulation of $\lbrack 1,2,3,4\rbrack = \langle h_{12},h_{23},h_{34}\rangle_+$ is represented by the sum of the Laplace transforms of its five simplicial cones,
	$$\langle h_{12},h_{13},h_{14}\rangle_+,\langle h_{13},h_{23},h_{14}\rangle_+, \langle h_{12},h_{34},h_{14}\rangle_+,\langle h_{14},h_{23},h_{24}\rangle_+,\langle h_{14},h_{24},h_{34}\rangle_+ ,$$
	as
	\begin{eqnarray*}
		\frac{1}{\left(y_1-y_2\right) \left(y_2-y_3\right) \left(y_3-y_4\right)} & = & \frac{1}{\left(y_1-y_2\right) \left(y_1-y_3\right) \left(y_1-y_4\right)}+\frac{1}{\left(y_1-y_3\right) \left(y_2-y_3\right) \left(y_1-y_4\right)}\\
		& + & \frac{1}{\left(y_1-y_2\right) \left(y_1-y_4\right) \left(y_3-y_4\right)} + \frac{1}{\left(y_2-y_3\right) \left(y_1-y_4\right) \left(y_2-y_4\right)}\\
		& + & \frac{1}{\left(y_1-y_4\right) \left(y_2-y_4\right) \left(y_3-y_4\right)},
	\end{eqnarray*}
	see Figures \ref{fig:triangulation-standard-plate-dual-associahedron-2} and \ref{fig:triangulationstandardplate}.
	Note that here we have neglected characteristic functions of common faces.  Indeed, we have the identity of characteristic functions, modulo characteristic functions of faces,
	\begin{eqnarray*}
		& & \lbrack\lbrack 1,2\rbrack\rbrack \bullet \lbrack\lbrack 2,3\rbrack\rbrack \bullet \lbrack\lbrack 3,4\rbrack\rbrack\\
		& = & \lbrack\lbrack 1,3\rbrack\rbrack\bullet \lbrack\lbrack 2,3\rbrack\rbrack \bullet  \lbrack\lbrack 1,4\rbrack\rbrack + \lbrack\lbrack 2,3\rbrack\rbrack \bullet \lbrack\lbrack 1,4\rbrack\rbrack \bullet \lbrack\lbrack 2,4\rbrack\rbrack \\
		&+& \lbrack\lbrack 1,2\rbrack\rbrack \bullet \lbrack\lbrack 1,4\rbrack\rbrack \bullet \lbrack\lbrack 3,4\rbrack\rbrack+\lbrack\lbrack 1,4\rbrack\rbrack \bullet \lbrack\lbrack 2,4\rbrack\rbrack  \bullet \lbrack\lbrack 3,4\rbrack\rbrack+ \lbrack\lbrack 1,2\rbrack\rbrack \bullet \lbrack\lbrack 1,3\rbrack\rbrack \bullet \lbrack\lbrack 1,4\rbrack\rbrack,
	\end{eqnarray*}
	where $\lbrack\lbrack i,j\rbrack\rbrack \bullet \lbrack\lbrack k,\ell\rbrack\rbrack$ equals the characteristic function of the conical hull $\langle e_i-e_j,e_k-e_\ell\rangle_+$, in the notation from \cite{EarlyCanonicalBasis}.
	
	Recall that the integral Laplace transform has in its kernel the span of all characteristic functions of higher codimension faces, as well as the characteristic function of any non-pointed cone (i.e. a cone which contains a double infinite line, such as the whole space).  For example, this gives a polytopal explanation for why the cyclic sum vanishes:
	$$\frac{1}{(y_1-y_2)\cdots (y_{m-1}-y_{m})}+\frac{1}{(y_2-y_3)\cdots (y_{m}-y_1)}+\cdots +\frac{1}{(y_{m}-y_1)\cdots (y_{m-1}-y_{m})}=0,$$
	that is, because it can be shown that the $n$ cyclically rotated simplicial cones $\lbrack 1,2,\ldots, n\rbrack$ have union the whole space and have non-intersecting interiors.  
	
	On the other hand, while characteristic functions of non-pointed cones are in the kernel of the functional representation $\lbrack\lbrack i,j\rbrack\rbrack\mapsto \sum_{a=0}^\infty e^{-(y_i-y_j)}=\frac{1}{1-x_i/x_j}$, where $x_i=e^{-y_i}$, the higher codimension faces are \textit{not}.  See for example \cite{BarvinokPammersheim} for details about functional representations of polyhedral cones.  In the case at hand, corresponding to the present triangulation we have the functional identity
	\begin{eqnarray*}
		&&	\frac{1}{\left(1-\frac{x_1}{x_3}\right) \left(1-\frac{x_2}{x_3}\right) \left(1-\frac{x_1}{x_4}\right)}+\frac{1}{\left(1-\frac{x_2}{x_3}\right) \left(1-\frac{x_1}{x_4}\right) \left(1-\frac{x_2}{x_4}\right)}+\frac{1}{\left(1-\frac{x_1}{x_2}\right) \left(1-\frac{x_1}{x_3}\right) \left(1-\frac{x_1}{x_4}\right)}\\
		&+&\frac{1}{\left(1-\frac{x_1}{x_2}\right) \left(1-\frac{x_1}{x_4}\right) \left(1-\frac{x_3}{x_4}\right)}+\frac{1}{\left(1-\frac{x_1}{x_4}\right) \left(1-\frac{x_2}{x_4}\right) \left(1-\frac{x_3}{x_4}\right)}\\
		&-&\left(\frac{1}{\left(1-\frac{x_1}{x_3}\right) \left(1-\frac{x_1}{x_4}\right)}+\frac{1}{\left(1-\frac{x_2}{x_3}\right) \left(1-\frac{x_1}{x_4}\right)}+\frac{1}{\left(1-\frac{x_1}{x_4}\right) \left(1-\frac{x_2}{x_4}\right)}+\frac{1}{\left(1-\frac{x_1}{x_4}\right) \left(1-\frac{x_3}{x_4}\right)}\right.\\
		&+& \left.\frac{1}{\left(1-\frac{x_1}{x_2}\right) \left(1-\frac{x_1}{x_4}\right)}\right)+\frac{1}{1-x_1/x_4}\\
		& =&\frac{1}{\left(1-\frac{x_1}{x_2}\right) \left(1-\frac{x_2}{x_3}\right) \left(1-\frac{x_3}{x_4}\right)},
	\end{eqnarray*}
	which is the Laplace transform of the standard plate $\lbrack 1,2,3,4\rbrack$.  Here the subtracted fractions are discrete Laplace transforms of the 2-dimensional simplicial cones respectively
	$$\langle h_{13},h_{14}\rangle_+, \langle h_{23},h_{14}\rangle_+, \langle h_{14},h_{24}\rangle_+, \langle h_{14},h_{34}\rangle_+,\langle h_{12},h_{14}\rangle_+,$$
	and the fraction in the last line is identified with the central ray $\langle h_{14}\rangle_+$, seen in Figures \ref{fig:triangulation-standard-plate-dual-associahedron-2} and \ref{fig:triangulationstandardplate} as the central vertex labeled $14$, which is shorthand for $\langle e_1-e_4\rangle_+$.
	
	Let us make a comparison between the expression above and the formula in Section 4.3 of \cite{MizeraTwistedCycles} for the generic diagonal element of the KLT matrix:
	\begin{eqnarray*}
		\langle C(12345),C(12345)\rangle & = & 1-\left(\frac{1}{1-e^{2\pi is_{12}}} + \frac{1}{1-e^{2\pi is_{23}}} + \frac{1}{1-e^{2\pi is_{34}}} + \frac{1}{1-e^{2\pi is_{45}}} + \frac{1}{1-e^{2\pi is_{51}}}\right)\\
		& + & \frac{1}{(1-e^{2\pi is_{12}})(1-e^{2\pi is_{34}})} + \frac{1}{(1-e^{2\pi is_{23}})(1-e^{2\pi is_{45}})}\\
		& + & \frac{1}{(1-e^{2\pi is_{34}})(1-e^{2\pi is_{51}})} + \frac{1}{(1-e^{2\pi is_{45}})(1-e^{2\pi is_{12}})}\\
		& + & \frac{1}{(1-e^{2\pi is_{51}})(1-e^{2\pi is_{23}})}.
	\end{eqnarray*}
	See also \cite{Frost} for an inclusion/exclusion derivation of terms in the diagonal of the inverse KLT matrix.
\end{example}
We construct an explicit identification, using the identity $s_{45} = s_{123}$ and its permutations to eliminate the label $5$.  This gives 
\begin{eqnarray*}
	\langle C(12345),C(12345)\rangle & = & 1-\left(\frac{1}{1-e^{2\pi is_{12}}} + \frac{1}{1-e^{2\pi is_{23}}} + \frac{1}{1-e^{2\pi is_{34}}} + \frac{1}{1-e^{2\pi is_{123}}} + \frac{1}{1-e^{2\pi is_{234}}}\right)\\
	& + & \frac{1}{(1-e^{2\pi is_{12}})(1-e^{2\pi is_{34}})} + \frac{1}{(1-e^{2\pi is_{23}})(1-e^{2\pi is_{123}})}\\
	& + & \frac{1}{(1-e^{2\pi is_{34}})(1-e^{2\pi is_{234}})} + \frac{1}{(1-e^{2\pi is_{123}})(1-e^{2\pi is_{12}})}\\
	& + & \frac{1}{(1-e^{2\pi is_{234}})(1-e^{2\pi is_{23}})}.
\end{eqnarray*}
The termwise identification here is as follows:
\begin{eqnarray*}
	1  \Leftrightarrow  \frac{1}{1-x_1/x_4}\ \ \ \ \ \ \ \ \ \ \ \ \  & & \\
	\frac{1}{1-e^{2\pi is_{12}}}  \Leftrightarrow  \frac{1}{1-x_1/x_2}\frac{1}{1-x_1/x_4}&& \frac{1}{1-e^{2\pi is_{12}}}\frac{1}{1-e^{2\pi is_{34}}} \Leftrightarrow \frac{1}{(1-x_1/x_2)(1-x_3/x_4)(1-x_1/x_4)}\\
	\frac{1}{1-e^{2\pi is_{23}}}  \Leftrightarrow  \frac{1}{1-x_2/x_3}\frac{1}{1-x_1/x_4}&&\frac{1}{1-e^{2\pi is_{23}}}\frac{1}{1-e^{2\pi is_{123}}} \Leftrightarrow \frac{1}{(1-x_2/x_3)(1-x_1/x_3)(1-x_1/x_4)}\\
	\frac{1}{1-e^{2\pi is_{34}}}  \Leftrightarrow  \frac{1}{1-x_3/x_4}\frac{1}{1-x_1/x_4}&&\frac{1}{1-e^{2\pi is_{34}}}\frac{1}{1-e^{2\pi is_{234}}} \Leftrightarrow \frac{1}{(1-x_3/x_4)(1-x_2/x_4)(1-x_1/x_4)}\\
	\frac{1}{1-e^{2\pi is_{123}}}  \Leftrightarrow  \frac{1}{1-x_1/x_3}\frac{1}{1-x_1/x_4}&&\frac{1}{1-e^{2\pi is_{123}}}\frac{1}{1-e^{2\pi is_{12}}} \Leftrightarrow \frac{1}{(1-x_1/x_3)(1-x_1/x_2)(1-x_1/x_4)}\\
	\frac{1}{1-e^{2\pi is_{234}}}  \Leftrightarrow  \frac{1}{1-x_2/x_4}\frac{1}{1-x_1/x_4}&&\frac{1}{1-e^{2\pi is_{234}}}\frac{1}{1-e^{2\pi is_{23}}} \Leftrightarrow \frac{1}{(1-x_2/x_4)(1-x_2/x_3)(1-x_1/x_4)}\\
\end{eqnarray*}

\begin{rem}
Dualizing term-by-term the formula given in Lemma 4.1 in \cite{MizeraTwistedCycles} for the self-intersection number, 
$$\langle C(1,2,\ldots, m),C(1,2,\ldots, m)\rangle = \sum_{k=0}^{m-2}\sum_{F=H_1\cap\cdots\cap H_k}(-1)^{m-2-k}\frac{1}{(1-e^{2\pi i{H_1}})\cdots (1-e^{2\pi i {H_k}})},$$
then we obtain exactly the general formula for the expansion as a signed sum of fractions.  Indeed, we have the alternating sum over the faces of the triangulation of 
$$\langle e_1-e_2,\ldots, e_{m-1}-e_{m}\rangle_+,$$ 

	$$\frac{1}{(1-x_1/x_2)\cdots (1-x_{m-1}/x_{m})}=\sum_{k=1}^{m-1}(-1)^{m-1-k}\sum_{\{(i_1,j_1),\ldots, (i_k,j_k),(1,m)\}} \frac{1}{(1-x_{i_1}/x_{j_1})\cdots (1-x_{i_k}/x_{j_k})},$$
labeled by the set of partial trees $\text{Tree}_{\le}^M$.
\end{rem}

In Section \ref{sec:mAlpha} we formalize the above discussion.

\section{Restricting biadjoint scalar amplitudes to subspaces of the kinematic space}\label{sec:mAlpha}
In this section, we compute the restriction of the biadjoint scalar amplitude \cite{CHY2014}, see also \cite{worldsheet}, to an interesting subspace $X^n$ of the kinematic space.  On this subspace, almost all of the poles become spurious and the amplitude beautifully collapses to a simple fraction.

Such restrictions have been considered more recently, since this paper was posted to the arXiv, for example in \cite{CE2021} in the context of generalized biadjoint scalar amplitudes \cite{CEGM2}, and in the context of $\text{Tr}(\phi^3)$ amplitudes in \cite{HiddenZeros}.

Denote by $X^n$ the subspace\footnote{In the notation of the rest of the paper, $X^n$ is a subspace of $\mathcal{K}^{n-2}$} of all symmetric (real) $n\times n$ matrices with 
\begin{enumerate}
	\item $s_{ii} = 0$ for $i=1,\ldots, n-1$.
	\item $s_{ij} = 0$ whenever $ \vert i-j\vert >1$, for $1\le i,j\le n-1$.
	\item $s_{in} = -\sum_{j=1}^{n-1} s_{ij}$ for each $i=1,\ldots, n-1$.
	\item $s_{nn} = \sum_{1\le i<j\le n-1}s_{ij}$ $(=s_{12\cdots n-1})$ (``massive $n^\text{th}$ particle'').
\end{enumerate}
While last condition (4) of the off-shell $n^\text{th}$ particle emerged a posteriori, it is tempting to study it in the context of the massive scattering equations \cite{Naculich}.  Alternatively, in more recent work \cite{CE2021} a variant on our kinematic subspace, called Minimal Kinematics, was introduced in which all particles remain massless.  See also \cite{EPS2024}.

For example, a generic element of $X^6$ looks like
$$\begin{bmatrix}
	0 & s_{12} & 0 & 0 & 0 & -s_{12} \\ 
	s_{12} & 0 & s_{23} & 0 & 0 & -(s_{12}+s_{23}) \\ 
0 & s_{23} & 0 & s_{34} & 0 & -(s_{23}+s_{34}) \\ 
0 & 0 & s_{34} & 0 & s_{45} & -(s_{34}+s_{45}) \\ 
0 & 0 & 0 & s_{45} & 0 & -s_{45} \\ 
-s_{12} & -(s_{12}+s_{23}) & -(s_{23}+s_{34}) & -(s_{34}+s_{45}) & -s_{45} & \sum_{i=1}^4 s_{i,i+1}
\end{bmatrix} $$

Our main result is to prove the following identities for the biadjoint amplitude and its $\alpha'$ deformation; notation will be explained subsequently.

\begin{thm}\label{thm:biadjoint identities}
	On the support of $X^n$, the amplitudes $m$ and $m_{\alpha'}$ simplify to respectively 
	\begin{eqnarray*}
		m((12\cdots n)\vert (12\cdots n))\big\vert_{X^n} & = & \frac{s_{12\cdots n-1}}{s_{12}s_{23}\cdots s_{n-2,n-1}}\\
		m_{\alpha'}(PT(1,\ldots, n),PT(1,\ldots, n))\big\vert_{X^n} & = & \frac{1-\exp\left(-2\pi i\alpha' \sum_{i=1}^{n-2}s_{i,i+1}\right)}{\prod_{i=1}^{n-2}\left(1-e^{-2\pi i\alpha' s_{i,i+1}}\right)}.
	\end{eqnarray*}
\end{thm}
The proof will involve the following change of variable, to transform the computation into an application of inclusion exclusion for simplicial cones: gathering together the independent coordinates of any $(s)\in X^n$ in an $(n-2)$-tuple $(s_{12},s_{23},\ldots, s_{n-2,n-1})$, let $(y_1,\ldots, y_n)\in\mathbb{R}^n$ such that $s_{i,i+1}=y_i-y_{i+1}$ for each $i=1,\ldots, n-2$.  Note that $(y_1,\ldots, y_n)$ is defined up to translation by a multiple of $(1,1,\ldots, 1)$.
	
\begin{rem}
	It is interesting to note that the biadjoint amplitude is proportional the mass of the $n^\text{th}$ particle, since by momentum conservation we have $s_{12\cdots n-1} = \left(\sum_{i=1}^{n-1}p_i\right)^2=p_n^2$.  Moreover, for the regular biadjoint amplitude, since 
	$$s_{12\cdots n-1} = s_{12}+s_{23}+\cdots+s_{n-2,n-1}$$
	we have the intriguing simplification
$$m((12\cdots n)\vert (12\cdots n))\big\vert_{X^n} = \sum_{i=1}^{n-2}\frac{1}{s_{12}s_{23}\cdots \widehat{s_{i,i+1}}\cdots s_{n-2,n-1}},$$
	where the terms in the summation are in bijection with the $(n-2)$ boundary \textit{facets} of the polyhedral cone
	$$\langle e_1-e_2,e_2-e_3,\ldots, e_{n-2}-e_{n-1}\rangle_+ = \left\{\sum_{i=1}^{n-2}t_i(e_i-e_{i+1}):t_i\ge 0\right\}$$
	from above, see Figures \ref{fig:triangulation-standard-plate-dual-associahedron-2} and \ref{fig:triangulationstandardplate}.  Here the notation $\widehat{s_{i,i+1}}$ indicates that the term $s_{i,i+1}$ has been omitted from the product.
\end{rem}

Taking into account the change of variable, the first identity in Theorem \ref{thm:biadjoint identities} becomes
\begin{eqnarray*}
	m((12\cdots n)\vert (12\cdots n))\big\vert_{X^n} & = &  \mathcal{L}_{\int}(\langle e_1-e_2,\ldots, e_{n-2}-e_{n-1}\rangle_+)\cdot (y_{1}-y_{n-1})\\
	& = & \sum_{i=1}^{n-2}\mathcal{L}_{\int}(\langle e_1-e_2,\ldots,\widehat{e_{i}-e_{i+1}}, \ldots, e_{n-2}-e_{n-1}\rangle_+),
\end{eqnarray*}
the sum of the integral Laplace transforms over the boundary components of the cone 
$$\langle e_1-e_2,\ldots, e_{n-2}-e_{n-1}\rangle_+=\{t_1(e_1-e_2)+\cdots +t_{n-1}(e_{n-2}-e_{n-1}):t_i\ge 0\},$$
where for the integrals the measure is taken on the respective ambient subspace.

Now let $\partial^k(\pi_\sigma)$ denote the set of faces of dimension $k$.

Then we have
\begin{eqnarray*}
	& & m_{\alpha'}((12\cdots n),(12\cdots n)))\big\vert_{X^n} =  \mathcal{L}_{\sum}(\langle e_1-e_2,\ldots, e_{n-2}-e_{n-1}\rangle_+)\cdot (1-e^{-(y_{1}-y_{n-1})})\\
	& = & \sum_{J\subsetneq \{1,2,\ldots, n-1\}} (-1)^{(n-2)-\vert J\vert}\mathcal{L}_{\sum}\left(\left \langle \{e_{j}-e_{j+1}:j\in J\} \right\rangle\right),
\end{eqnarray*}
which is an application of inclusion/exclusion to the set of all faces of codimension at least $1$ the cone
$$\langle e_1-e_2,\ldots, e_{n-2}-e_{n-1}\rangle_+.$$

Let $\mathcal{T}_n$ be the set of all (partial) triangulations of a polygon $\text{poly}_n$ oriented counterclockwise with vertices $1,2,\ldots, n$; these partial triangulations are in correspondence with faces of the associahedron, or dually with the faces of the simplicial complex formed from the standard triangulation of the simplicial cone
$$\langle e_1-e_2,\ldots, e_{n-2}-e_{n-1}\rangle_+.$$
This can be seen by taking the cone over the triangulation of the so-called root polytope from \cite{GelfandGraevPostnikov}, see also \cite{Cho} for a proof of the triangulation by adding relative volumes of simplices.

For example, the set of forests each having a single tree consists of the $C_n$ binary trees, where $C_n = \frac{1}{n-1}\binom{2(n-2)}{2}$ is the Catalan number; these label the maximal dimension simplicial cones in the triangulation.

Let us express the formula for $\langle C(1,2,\ldots, n),C(1,2,\ldots, n)\rangle$ from \cite{MizeraTwistedCycles} in terms of partial triangulations.  We recall that
\begin{eqnarray}\label{eqn:Mizera m alpha}
\langle C(1,2,\ldots, n),C(1,2,\ldots, n)\rangle = \sum_{T\in\mathcal{T}_n}(-1)^{(n-3)-\vert T\vert}\prod_{E\in T}\frac{1}{1-e^{-2\pi i\alpha's_{\overline{E}}}}
\end{eqnarray}
as $T$ ranges over all partial triangulations of $\text{poly}_n$ and $E\in T$ ranges over the edges of $T$, and where $\vert T \vert$ is the number of edges $E$ of $T$.  Here $s_{\overline{E}} = s_{i i+1\cdots j}$ if $E=(i,j)$.

We observe that the restriction of the right-hand side of Equation \eqref{eqn:Mizera m alpha} to $X^n$ gives rise to a termwise identification with the sums of the Laplace transforms of the boundary components of the cone 
$$\langle e_1-e_2,\ldots, e_{n-2}-e_{n-1}\rangle_+=\{t_1(e_1-e_2)+\cdots +t_{n-1}(e_{n-2}-e_{n-1}):t_i\ge 0\}.$$

Given an n-cycle $\sigma=(\sigma_1,\sigma_2,\ldots, \sigma_n)$ with minimal element $\sigma_1$, let use define $\pi_\sigma = \langle e_{\sigma_1}-e_{\sigma_2},\ldots, e_{\sigma_{n-1}}-e_{\sigma_n}\rangle_+$.  Denote by 
$$\partial(\pi_\sigma) =\left\{ \langle e_{\sigma_1}-e_{\sigma_2},\ldots,\widehat{e_{\sigma_i}-e_{\sigma_{i+1}}},\ldots, e_{\sigma_{n-1}}-e_{\sigma_n}\rangle : i=1,\ldots, n-1\right\}$$
the boundary of $\pi_\sigma$.
\begin{prop}\label{prop:triangulation Nearest Neighbor}
	With $\sigma=(1,2,\ldots, n)$, we have the simplification
	$$\langle C(1,2,\ldots, n),C(1,2,\ldots, n)\rangle\big\vert_{X^n} = \sum_{\pi\in \partial(\pi_{\sigma})} \mathcal{L}_{\sum}(\pi).$$
\end{prop}

\begin{proof}
	
	For $i=1,\ldots, n-2$, set $2\pi i\alpha's_{i,i+1} = y_i-y_{i+1}$, where the point $(y_1,\ldots, y_n)$ is defined only modulo translation by multiples of $(1,\ldots, 1)$.
	We show that 
	$$\langle C(1,2,\ldots, n),C(1,2,\ldots, n)\rangle\big\vert_{X^n} = \frac{1-e^{-2\pi i\alpha' s_{1,n-1}}}{\prod_{i=1}^{n-2}\left(1-e^{-2\pi i\alpha' s_{i,i+1}}\right)}.$$
	
	For any contiguous subset $\{i,i+1,\ldots, j\}$ we have telescopically $s_{i,i+1,\ldots, j} = y_i-y_j$, hence
	\begin{eqnarray*}
		& & \langle C(1,2,\ldots, n),C(1,2,\ldots, n)\rangle\big\vert_{X_n(\mathbf(0))}\\
		& = &  \sum_{T\in\mathcal{T}_n}(-1)^{(n-3)-\vert T\vert}\prod_{E\in T}\frac{1}{1-e^{-2\pi i\alpha' s_{\overline{E}}}}\\
		& = & \sum_{T\in\mathcal{T}_n}(-1)^{(n-3)-\vert T\vert}\prod_{(i,j)\in T}\frac{1}{1-e^{- y_{ij}}}\\
		& = & \sum_{T\in\mathcal{T}_n}(-1)^{(n-3)-\vert T\vert}\prod_{(i,j)\in T}\left(\sum_{m=0}^\infty e^{-m(y_{i}-y_{j})}\right)\\
		& = & \sum_{T\in\mathcal{T}_n}(-1)^{(n-3)-\vert T\vert}\left(\sum_{m_1,\ldots, m_{\vert T\vert}=0}^\infty e^{-\left(m_1(y_{i_1}-y_{j_1})+\cdots m_{\vert T\vert}(y_{i_{\vert T\vert}}-y_{j_{\vert T\vert}})\right)}\right)\\
		& = & \sum_{(\{i_a,j_a\}=T)\in\mathcal{T}_n}(-1)^{(n-1)-\vert T\vert}\mathcal{L}_{\sum}\left(\langle e_{i_1}-e_{j_1},\ldots, e_{i_{\vert T\vert}}-e_{j_{\vert T\vert}},e_1-e_{n-1}\rangle_+ \right)\left(\frac{1}{1-e^{-(y_1-y_{n-1})}}\right)^{-1}.
	\end{eqnarray*}
	Here 
	$$\mathcal{L}_{\sum}\left(\langle e_{i_1}-e_{j_1},\ldots, e_{i_{\vert T\vert}}-e_{j_{\vert T\vert}},e_1-e_{n-1}\rangle_+\right)=\sum_{m_0,m_1,\ldots, m_{\vert T\vert}=0}^\infty e^{-\left(m_0(y_{1}-y_{n-1})+m_1(y_{i_1}-y_{j_1})+\cdots m_{\vert T\vert}(y_{i_{\vert T\vert}}-y_{j_{\vert T\vert}})\right)}$$
	is the discrete Laplace transform of the polyhedral cone $\langle e_{i_1}-e_{j_1},\ldots, e_{i_{\vert T\vert}}-e_{j_{\vert T\vert}},e_1-e_{n-1}\rangle_+$, where the series converges on the union of Weyl chambers where 
	$y_{i_a}>y_{j_a}$ for all $a$, and $y_1>y_{n-1}$.  Now we recognize 
	$$ \sum_{(\{i_a,j_a\}=T)\in\mathcal{T}_n}(-1)^{(n-1)-\vert T\vert}\mathcal{L}_{\sum}\left(\langle e_{i_1}-e_{j_1},\ldots, e_{i_{\vert T\vert}}-e_{j_{\vert T\vert}},e_1-e_{n-1}\rangle_+ \right)$$
	as the image of the Laplace transform valuation of the alternating sum of characteristic functions over the set of faces of simplicial cones in the triangulation of $\langle e_1-e_2,\ldots, e_{n-2}-e_{n-1}\rangle_+$.

	We therefore obtain for the Laplace transform
	$$\mathcal{L}\left(\langle e_{1}-e_{2},\ldots, e_{n-2}-e_{n-1}\rangle_+\right)=\left(\prod_{i=1}^{n-2}\frac{1}{1-e^{-(y_i-y_{i+1})}}\right),$$
	hence
	\begin{eqnarray*}
		& & \langle C(1,2,\ldots, n),C(1,2,\ldots, n)\rangle\big\vert_{X^n}\\
		& = & \left(\prod_{i=1}^{n-2}\frac{1}{1-e^{-(y_i-y_{i+1})}}\right)\left(\frac{1}{1-e^{-(y_1-y_{n-1})}}\right)^{-1}=\frac{1-e^{-2\pi i\alpha' s_{1,n-1}}}{\prod_{i=1}^{n-2}\left(1-e^{-2\pi i\alpha' s_{i,i+1}}\right)}\\
	\end{eqnarray*}
	
\end{proof}

\begin{prop}
	Let real numbers $y_1>y_2>\cdots>y_{n-1}$ be given.  We assume that the first n-1 particles are massless, but that particle $n$ has a positive mass $p_n^2 = s_{12\cdots n-1}$, and that for all $1\le i<j\le n-1$	only the adjacent Mandelstam variables are nonzero: we define
	$$s_{i,i+1}=s_{i+1,i}=y_i-y_{i+1},$$ and put $s_{ij}=0$ whenever $\vert i-j\vert >1$.  For each $i=1,\ldots, n-1$, define 
	$$s_{in}=-\sum_{j=1}^{n-1} s_{ij}.$$
	
	Then
	$$m(PT(1,2,\ldots, n),PT(1,2,\ldots, n))\big\vert_{X^n} = \frac{s_{12\cdots n-1}}{s_{12}\cdots s_{n-2 n-1}} = \frac{p_n^2}{s_{12}\cdots s_{n-2 n-1}}.$$
\end{prop}
\begin{proof}
	This is precisely analogous to the proof of Proposition \ref{prop:triangulation Nearest Neighbor}, except that in the continuous Laplace transform higher codimension faces are mapped to zero, and we end up with only the Catalan-many fractions in Mandelstam variables, which by \cite{GelfandGraevPostnikov,Cho} simplify over a common denominator and then decompose, as 
	$$\frac{y_{1}-y_{n-1}}{y_{12}\cdots y_{n-2,n-1}} = \sum_{i=1}^{n-2} \frac{1}{y_{12}\cdots \widehat{y_{i,i+1}}\cdots y_{n-2,n-1}},$$
	the sum over the $n-2$ boundary components of $\pi_{(1,2,\ldots, n-1)}$.
\end{proof}

\section{Additional remarks}
In this note, we classified the set of generalized permutohedra that are determined by the facet inequalities $s_{aJ} = s_{a\vert J} -c_{J}\ge 0$, that is $s_{a\vert J}\ge c_J$, where $c_J = \sum_{i,j\in J: i<j} c_{ij}$ and $c_{ij}$ are the $\binom{n}{2}$ constants and $J$ varies over all proper nonempty subsets of $\{1,\ldots, n\}$.  The main result was to prove the face distances $s_{aJ}\ge 0$ determine a \textit{zonotopal} generalized permutohedron.  In fact, the constants $c_{ij}$ are the dilation factors for the root directions.  Let us point out that as the parameters $c_{ij}$ vary, we obtain the \textit{edge}-deformation cone, as discussed in the Appendix of \cite{FacesPermutohedra}, see in particular Definition 15.1.

In Sections \ref{sec: Appendix triangulations dual associahedra} and \ref{sec:mAlpha} we discussed the relationship between the triangulation of the permutohedral cone 
$$\langle e_1-e_2,\ldots, e_{n-1}-e_n\rangle_+$$
and the set of tangent cones to the associahedron.

It turns out that there is a similar dual interpretation of the cyclohedron which uses triangulations of the convex hull of \textit{all} roots $e_i-e_j$ parametrized by $n$-cycles.  Each such triangulation forms a complete fan; it is an example of a simplicial complex known as a blade.  Characteristic functions of such were studied in \cite{EarlyBlades}.

\section{Acknowledgements}
We thank Freddy Cachazo, Sebastian Mizera and Victor Reiner for interesting and helpful discussions.  We also thank the anonymous referee for careful reading and insightful comments.

\end{document}